\documentclass[hidelinks, 11 pt,a4paper]{article}
\usepackage[utf8]{inputenc}
\usepackage[swedish,english]{babel}
\usepackage{amsmath}
\usepackage{dsfont} 
\usepackage{amsthm}
\usepackage{enumerate}
\usepackage{ mathrsfs } 
\usepackage{amssymb }
\usepackage[all,knot]{xy}
\usepackage{amsthm}
\usepackage{dsfont} 
\usepackage{hyperref} 
\usepackage{enumitem} 
\usepackage[titletoc,title]{appendix}
 \usepackage{color}

\newcommand{\ot}{\otimes}
\newcommand{\Z}{\mathbb{Z}} 
\newcommand{\Q}{\mathbb{Q}}

\newcommand{\ubar}{U^{Bar}}

\newcommand{\op}{\oplus} 
\DeclareMathOperator{\id}{id}

\DeclareMathOperator{\Hom}{Hom}

\DeclareMathOperator{\charac}{char}
\DeclareMathOperator{\im}{im}
\renewcommand{\dots}{\ldots}
\newcommand{\dgl}{\boldsymbol{\mathrm{DGL}}_\k}
\newcommand{\dga}{\boldsymbol{\mathrm{DGA}}_\k} 
\newcommand{\cdga}{\boldsymbol{\mathrm{CDGA}}_\k}

\newcommand{\h}{\mathcal{H}}

\renewcommand{\k}{\mathds k}
\newcommand{\ainf}{A_\infty\text{-}}
\newcommand{\cinf}{C_\infty\text{-}}
\newcommand{\linf}{L_\infty\text{-}}
\newcommand{\p}{\mathscr P}
\newcommand{\Vpd}{\mathbb V_{\p^!}}

\newcommand{\cL}{\mathcal L}

\newtheorem{thm}{Theorem}[section]
\newtheorem{cor}[thm]{Corollary}
\newtheorem{prop}[thm]{Proposition}

\newtheorem{lemma}[thm]{Lemma} 
\theoremstyle{definition}
\newtheorem{dfn}[thm]{Definition}

\newtheorem{rmk}[thm]{Remark}

\newcommand{\Addresses}{{
  \bigskip
  \footnotesize
  
\textsc{Department of Mathematics, Stockholm University, SE-106 91 Stockholm, Sweden}\par\nopagebreak
  \textit{E-mail address:} \texttt{bashar@math.su.se}
}}

\begin{document}
\title{Non-commutative formality implies commutative and Lie formality} 
\author{Bashar Saleh} 
\date{}

\maketitle

\begin{abstract} 
 Over a field of characteristic zero we prove two formality conditions. We prove that a dg Lie algebra is formal if and only if its universal enveloping algebra is formal.  We also prove that a commutative dg algebra is formal as a dg associative algebra if and only if it is formal as a commutative dg algebra. We present some consequences of these theorems in rational homotopy theory.
\end{abstract}

\section{Introduction} Formality is an important concept in rational homotopy theory (\cite{deligne75}), deformation quantization (\cite{kontsevich03}), deformation theory (\cite{goldman88}), and other branches of mathematics where differential graded homological algebra is used. 
The notion of  formality  exists in many categories, e.g. the category of (commutative) dg associative algebras and the category of dg Lie algebras. An object  $A$ in such a category  is called formal if there exists a zig-zag of  quasi-isomorphisms  connecting $A$ with its cohomology $H(A)$;
\[A\xleftarrow\sim
B_1\xrightarrow\sim \cdots \xleftarrow\sim
B_n\xrightarrow\sim H(A).\] 

A functor between categories in which the notion of formality exists may or may not preserve formal objects. For example, over a field  of characteristic zero, it is known that the universal enveloping algebra functor $U\colon\dgl\to \dga$ preserves formal objects  (\cite[Theorem 21.7]{felixrht}). That means that the formality of a dg Lie algebra (dgl) $L$ implies the formality of $UL$ (as a dg associative algebra (dga)). But what about the reversed relation? Does the formality of $UL$ imply the formality of $L$? In this paper we show that this holds for dg Lie algebras over a
field of characteristic zero.

\begin{thm}\label{mainthmone} A dg Lie algebra  $L$  over a field of characteristic zero is formal  if and only if its universal enveloping algebra $UL$ is formal as a dga.
\end{thm}

Among the results in the spirit of Theorem \ref{mainthmone}, there is a theorem by Aubry and Lemaire (\cite{aubry88}) saying that two dgl morphisms $f,g\colon L\to L'$ are homotopic if and only if $U(f),U(g)\colon UL\to UL'$ are homotopic. The author do not think that the result by Aubry and Lemaire implies Theorem \ref{mainthmone} or vice versa.

Another remark is about the result by Milnor and Moore (\cite{milnor65}) saying that, over a field of characteristic zero, the universal enveloping algebra defines an equivalence of categories between the category of dg Lie algebras and the category of connected cocommutative dg Hopf algebras. This equivalence together with Theorem \ref{mainthmone} and with the fact that a dgl morphism $f\colon L\to L'$ is a quasi-isomorphism if and only if $U(f)\colon UL\to UL'$ is a quasi-isomorphism (\cite[Theorem 21.7 (ii)]{felixrht}), gives that a connected cocommutative dg Hopf algebra is formal  as a connected cocommutative dg Hopf algebra if and only if it is formal as a dga.

We demonstrate a topological consequence of Theorem \ref{mainthmone}. The rational homotopy type of a simply connected space $X$ is algebraically modeled by Quillen's dg Lie algebra $\lambda(X)$ over the rationals (\cite{quillen69}). The space $X$ is called coformal if $\lambda(X)$ is a formal dgl. It is known that  there exists a zig-zag of quasi-isomorphisms connecting $U \lambda(X)$  to the algebra $C_*(\Omega X,\Q)$ of singular chains on the Moore loop space of $X$ (\cite[Chapter 26]{felixrht}). From  Theorem \ref{mainthmone} the following corollary follows immediately:

\begin{cor} Let $X$ be a simply connected space. Then $X$ is coformal if and only if $C_*(\Omega X;\Q)$ is formal as a dga.
\end{cor}

Our second formality result is concerning the forgetful functor from the category of commutative dgas (cdgas) to the category of dgas. This functor  preserves formality; a cdga which is formal as a cdga is obviously formal as a dga. Again, we ask whether this relation is reversible or not. We will prove that over a field  of characteristic zero the answer is positive.

\begin{thm}\label{mainthmtwo}
Let $A$ be a cdga over a field of characteristic zero. Then $A$ is formal as dga if and only if it is formal as a cdga. 
\end{thm}
Recall that a space $X$ is called rationally formal if the Sullivan-de Rham algebra $A_{PL}(X;\Q)$ is formal as a cdga (\cite[Chapter 12]{felixrht}). In that case the rational homotopy type of $X$ is a formal consequence of its cohomology $H^*(X;\Q)$, meaning that $H^*(X;\Q)$ determines the rational homotopy type of $X$. Moreover, it is known that there exists a zig-zag of quasi-isomorphisms connecting $A_{PL}(X;\Q)$ with the singular cochain algebra $C^*(X;\Q)$ of $X$ (\cite[Theorem 10.9]{felixrht}).  An immediate topological consequence is the following corollary:
\begin{cor}
A space $X$ is rationally formal if and only if the singular cochain algebra $C^*(X;\Q)$ of $X$ is formal as a dga. 
\end{cor} 

\subsection*{Overview} 
The reader is assumed to be familiar with the theory of operads and with the notions of $A_\infty$-, $\cinf$, and $\linf$algebras.  We refer the reader to \cite{lodayoperad,markloperad,keller01} for introductions to these subjects.

In Section 2 we review  Baranovsky's universal enveloping construction on the category of $\linf$algebras (\cite{baranovsky08}). The construction  is a generalization of the universal enveloping algebra functor and is an important ingredient in the proof of Theorem \ref{mainthmone}. 
In Section 3 we present  an obstruction theory for formality in different categories. The obstructions will be cohomology classes of certain cohomology groups.
 The obstruction theory together with Baranovsky's universal enveloping will give us tools to compare the concept of formality in $\dga$ and $\dgl$  ($\charac\k=0$). This will be treated in Section 4 and will finally yield a proof of Theorem \ref{mainthmone}.
  In Section 5 we prove Theorem \ref{mainthmtwo}.
 
  The reader interested only in Theorem \ref{mainthmone} may skip Section  5, whilst the reader  only interested in Theorem \ref{mainthmtwo}  may skip Sections 2 and 4.

\subsection*{Conventions}
\begin{itemize}
\item $S_k$ denotes the symmetric group on  $k$ letters.

\item The Koszul sign of a permutation $\sigma\in S_k$ acting on $v_1\cdots v_k\in V^{\ot k}$ (where $V$ is a graded vector space) is given by the following rule: The Koszul sign of an adjacent transposition that permutes $x$ and $y$ is given by $(-1)^{|x||y|}$. This is then extended multiplicatively to all of $S_k$ (recall that the set of adjacent transpositions generates $S_k$).
\item The suspension $sV$ of a graded vector space $V$ is the graded vector space given by $sV^i=V^{i+1}$. The suspension of a cochain complex $(C,d)$ is the cochain complex $(sC,-sds^{-1})$.
\item  A standing assumption will be that $\k$ is  a field of characteristic zero. We will only consider (co)algebras and (co)operads over fields of characteristic zero. 
\end{itemize}

\begin{center}\textbf{Acknowledgments}\end{center} I would like to thank my advisor Alexander Berglund for  his invaluable guidance during the preparation of this paper  and  also for proposing the topics treated here.  I would also like to thank Stephanie Ziegenhagen for her careful reading of this paper and her comments and suggestions. Thanks to Kaj Börjeson and Felix Wierstra for introducing me to many concepts that I had very little knowledge about.
Finally, I would like to  thank Victor Protsak for helping me with the proof of Proposition \ref{lie-mod-incl} (at MathOverflow).

\section{Baranovsky's universal enveloping for $\linf$algebras}
The proof of Theorem \ref{mainthmone} will partly rely on a 
construction by Baranovsky (\cite{baranovsky08}) that generalizes 
the universal enveloping algebra construction to $\linf$algebras.

Applying Baranovsky's universal enveloping (denoted by $\ubar$) to an $\linf$algebra $(L,\{l_i\})$ gives an $\ainf$algebra
$\ubar(L,\{l_i\})= (\Lambda L, \{m_i\}) 
$  where $\Lambda L$ is the underlying graded vector space of the symmetric algebra on $L$. Applying $\ubar$ to an $\linf$morphism $\phi\colon L\to L'$ gives an $\ainf$morphism $\ubar(\phi)\colon\ubar(L) \to \ubar(L')$.

$\ubar$ is not a functor since it fails to preserve compositions (i.e. $\ubar(\psi\circ\phi) \neq \ubar(\psi)\circ\ubar(\phi)$ in general). However,  the restriction of $\ubar$ to $\dgl\subset \infty\text-\linf\boldsymbol{\mathrm{alg}}$ (here $\infty\text-\linf\boldsymbol{\mathrm{alg}}$ denotes the category of $\linf$algebras with $\infty$-morphisms), coincides with the usual universal enveloping algebra functor, denoted by $U$.

We record some properties of $\ubar$.

\begin{thm}\label{ubar} 
Let $(L,\{l_i\})$ be an $\linf$algebra with universal enveloping $\ubar(L,\{l_i\}) = (\Lambda L, \{m_i\})$. The following properties holds:
\begin{enumerate}[label = (\alph*)]
\item $m_1\colon \Lambda L\to \Lambda L$ is the symmetrization  of $l_1$ (i.e. $m_1 = \Lambda(l_1)$).
\item If $\phi\colon (L,\{l_i\})\to (L',\{l_i'\})$ is an $\linf$quasi-isomorphism, then $\ubar(\phi)\colon \break(\Lambda L,\{m_i\})\to (\Lambda L',\{m_i'\})$ is an $\ainf$quasi-isomorphism.
\item The map $m_j\colon    (\Lambda L)^{\ot j}\to \Lambda L$ depends only on $L,l_1,l_2,\dots,l_j$. In particular, if $(L,\{k_i\})$ is another $\linf$algebra structure on  the same vector space $L$ with $l_j=k_j$ for $j=1,2,\dots,d$, then we have that $\ubar(L,\{k_i\}) = (\Lambda L ,\{n_i\})$ with $n_j= m_j$ for $j=1,2,\dots,d$. 
\item
Let $v_1,\dots,v_j\in L\subset \ubar L$. Then 
\[l_j(v_1\cdots v_j) = \sum_{\sigma\in S_j}\gamma(\sigma;v_1,\dots,v_j)m_j(v_{\sigma^{-1}(1)}\cdots v_{\sigma^{-1}(j)}),\]
where $\gamma(\sigma;v_1,\dots,v_n)$ is the product of the sign of the permutation $\sigma$ and the Koszul sign obtained by applying $\sigma$ on $v_1\cdots v_j$. 
\item The restriction $\ubar|_{\dgl}$ of $\ubar$ to $\dgl$ coincides with the ordinary universal enveloping algebra functor.
\end{enumerate}
\end{thm}
Properties (a)-(c) are not explicitly stated in \cite{baranovsky08} so we will  briefly recall Baranovsky's construction in order to prove these properties.

\subsection*{A summary of the construction}
Given a complex $(V,d)$, let $T^*_a(V)$ ($T^*_c(V)$) and $\Lambda^*_a(V)$ ($\Lambda^*_c(V)$) denote the tensor  respective  symmetric (co)algebra on $V$   with (co)differential corresponding to the unique (co)derivation extension of $d$.

Let $(L,\{l_i\})$ be an $\linf$algebra. We start by considering the complex $(L,l_1)$ and construct from it  two coalgebras; $(T^*_c(s\overline{\Lambda_a^*( L)}), d^\circ)$ and $(T^*_c(s\overline{\Omega(\Lambda_c^*(sL))}),\delta^\circ)$ (where  $\Omega$ denotes the cobar construction and $\overline{(\cdot)}$ denotes the augmentation ideal).

Baranovsky shows that there exists a coalgebra contraction from  $T^*_c(s\overline{\Omega(\Lambda_c^*(s L))})$ to $T^*_c(s\overline{\Lambda_a^*(L)})$
\begin{equation}\label{contraction}\xymatrix{**[r]{T^*_c(s\overline{\Omega(\Lambda_c^*( sL))})}\ar @`{(-12,-12),(-12,12)}  ^H\ar@<3 pt>[rr]^{\qquad\qquad\qquad F}  &&**[r]T^*_c(s\overline{\Lambda_a^*(L)})\ar@<2 pt>[ll]^{\qquad\qquad\qquad G}}.\end{equation}
By comparing $T^*_c(s\overline{\Omega(\Lambda_c^*( sL))})$ with the cobar-bar construction on the Chevalley-Eilenberg construction on the $\linf$algebra $(L,\{l_i\})$, denoted by $B\Omega C(L)$, we see that they only differ by their differentials. The differential $\delta$ of $B\Omega C(L)$ is given by
\[\delta = \delta^\circ + t_\mu+ t_L,\] 
where $t_\mu$ is the part that encodes the multiplication on  
$\Omega C(L)$ and where $t_L = t_2+t_3+\dots$ encodes the $
\linf$structure on $L$ with $t_i$ encoding $l_i$. Applying the basic perturbation lemma  to the perturbation $t_\mu+t_L$ of the 
contraction above results in a new differential $d = (d^
\circ)_{t_\mu+t_L}$ on $T^*_c(s\overline{\Lambda^*_a(L)})$ 
which corresponds to an $\ainf$algebra structure on $\Lambda 
L$, which will be Baranovsky's universal enveloping $\ubar(L,
\{l_i\})$.

\subsection*{Geometric grading}
Baranovsky introduces a geometric grading on $B\Omega C(L)$ by first declaring that an element of $s^{-1}\Lambda^{k}_c(sL)$ is of geometric degree $k-1$ and  extends then the grading to $B\Omega C(L)$  by the following rule; the  geometric degree of $\alpha\ot\beta$  is the sum of the geometric degrees of $\alpha$ and $\beta$.
The maps in the contraction \eqref{contraction} and the perturbations $t_\mu$ and $t_L = t_2+t_3+\dots$ satisfy some conditions regarding the geometric grading:
\begin{itemize}
\item The image of $G$ belongs to the geometric degree 0 part.
\item $H$ increases the geometric degree by 1.
\item $t_\mu$ preserves the geometric degree.
\item $t_i$ decreases the geometric degree by $i-1$ and vanishes on elements of geometric degree $< i-1 $
\end{itemize}

\subsection*{Proof of Theorem \ref{ubar}}
\begin{proof}
By the basic perturbation lemma (the lemma is stated in \cite[Lemma 2]{baranovsky08})  we have that the differential $d= (d^\circ)_{t_\mu+t_L}$ is given by 
\begin{equation*}d = d^\circ + F\left(\sum_{i\geq0} ((t_\mu+t_L)H)^i\right)(t_\mu+t_L) G.\end{equation*}
Since the image of $G$ belongs to the geometric degree 0 part and since $t_L= t_2+t_3+\dots$ vanishes on elements of geometric degree 0, we may rewrite the differential as
\begin{equation}\label{transferedpert}
d = d^\circ + F\left(\sum_{i\geq0} (t_\mu H+t_2H+t_3H+\dots)^i\right)t_\mu G.
\end{equation} 

The terms in the differential above that correspond to $m_n\colon    \ubar(L)^{\ot n}\to \ubar(L)$ are those terms that contain $t_\mu$ exactly $(n-1)$ times (see the proof of \cite[Theorem 3]{baranovsky08} for details).
\\\\
(a) Since $d^\circ$ is the only term in \eqref{transferedpert} that does not contain $t_\mu$ as a factor, we have that $d^\circ$ is the part of the differential $d$ that corresponds to  $m_1\colon    \ubar(L)\to\ubar(L)$. One can easily see that $d^\circ$ corresponds to $\Lambda(l_1)\colon    \Lambda L \to \Lambda L$.
\\\\
(b) By \cite[Theorem 3.i]{baranovsky08} we have that the first component $ \ubar(\phi)_1$ of $\ubar(\phi)$ is given by $\Lambda(\phi_1)$, where $\phi_1$ is the first component of $\phi$. In order to show that $\ubar(\phi)$ is an $\ainf$quasi-isomorphism, we need to show that 
\begin{equation}\label{symmetrization}\ubar(\phi)_1 = \Lambda(\phi_1) \colon    (\Lambda L,m_1)\to (\Lambda L',m_1')\end{equation}
 is a quasi-isomorphism of complexes.  Since $\phi$ is an $\linf$quasi-isomorphism, it follows that $\phi_1\colon   (L,l_1)\to (L',l_1')$ is a quasi-isomorphism  of complexes. By (a), $m_1$ and $m_1'$ are given by symmetrizations of $l_1$ and $l_1'$ respectively, which means that that the map in \eqref{symmetrization} is obtained by applying the symmetrization functor $\Lambda(-)$ on $\phi_1\colon    (L,l_1)\to (L',l_1')$. Over a field $\k$ of characteristic zero we have that the symmetrization functor $\Lambda(-)$ preserves quasi-isomorphisms (since $L\ot_\k-$ is exact and that taking $S_n$-coinvariants is also exact), and (b) follows.
\\\\
(c) Firstly, $H$ depends only on $L$ and $l_1$ by \cite[Theorem 1]{baranovsky08}. Moreover, we have that $t_\mu H$ increases the geometric degree by 1 while $t_iH$ decreases the geometric degree by $(i-2)$. Furthermore, $t_iH$ vanish on elements of degree $<i-2$. That means if there exists a non-zero term containing $t_iH$, then $t_\mu H$ has to occur at least $(i-2)$ times before $t_iH$ (i.e. to the right  of $t_i H$). 

We have that $m_n$ corresponds to those non-zero terms that contain $t_\mu$ exactly $(n-1)$ times, which is equivalent to  those terms that contain contain $t_\mu H$ exactly $(n-2)$ times. These terms cannot contain any $t_i H$ where $i>n$ (since they are non-zero). 
From this and the fact that $t_i$ is completely encoded by $l_i$, claim (c) follows. 
\\\\
(d) See \cite[Theorem 3.vii]{baranovsky08}
\\\\
(e) See \cite[Theorem 3.v]{baranovsky08}
\end{proof}

\section{Minimal $\p_\infty$-algebras and obstructions to formality} 
Given an algebraic operad $\p$, we have that the cohomology of a dg $\p$-algebra has an induced dg $\p$-algebra structure with a trivial differential (\cite[Proposition 6.3.5]{lodayoperad}). Thus, the notion of formality makes sense in the category of dg $\p$-algebras. 

If $\mathscr P$ is a Koszul operad, we denote the operad obtained by applying the cobar construction on the Koszul dual cooperad of $\p$ by $\p_\infty$ (\cite[Chapter 10]{lodayoperad}). The category of $\p_\infty$-algebras with $\p_\infty$-morphisms (denoted by $\infty\text-\p_\infty\text-\boldsymbol{\mathrm{alg}}$) contains the category of $\p$-algebras as a subcategory and has some properties that the category of $\p$-algebras lacks, e.g. that quasi-isomorphisms are invertible up to homotopy. 
\begin{thm}[{\cite[Theorem 11.4.9]{lodayoperad}}]
Let $\p$ be a Koszul operad over a field of characteristic zero and let $A$ be a dg $\p$-algebra. Then  $A$ is formal as a $\p$-algebra if and only if there exists a $\p_\infty$-algebra quasi-isomorphism $A\to H(A)$. 
\end{thm}
In this paper we will be interested in algebras over the operads $\mathscr Ass$, $\mathscr Com$, and $\mathscr Lie$, which are all Koszul. From now on, $\p$ is either $\mathscr Ass$, $\mathscr Com$ or $\mathscr Lie$, which means that a dg $\p$-algebras is either  a dga, cdga, or dgl, and that a $\p_\infty$-algebra is either an $\ainf$, $\cinf$, or $\linf$algebra.

We denote the Koszul dual operad of $\p$ by $\p^!$ (recall that $\mathscr Ass^! = \mathscr Ass$, $\mathscr Com^! = \mathscr Lie$, and $\mathscr Lie^! =\mathscr Com$). 
We have that a $\p_\infty$-algebra structure on  a vector  space $A$ is a collection $(A,\{b_n\})$ where $b_n\colon    \p^!(n)\ot_{S_n}A^{\ot n}\to A$, $n\geq 1$, are linear maps of degree $n-2$ that satisfy certain compatibility conditions (see \cite{lodayoperad}). 
A dg $\p$-algebra $(A,b_1,b_2)$  may be regarded as $\p_\infty$-algebra by identifying $(A,b_1,b_2)$ with $(A,b_1,b_2,0,0,\dots)$.
A morphism of $\p_\infty$-algebras $\phi\colon    (A,\{b_n\})\to (A',\{b_n'\})$, is a collection $\phi = (\phi_n)$ where $\phi_n$ are maps $\p^!(n)\ot A^{\ot n}\to A'$ of degree $n-1$ that satisfy certain conditions.

Given an operad $\p$ there is the notion of the operadic 
cochain complex $C^*_{\p}(A)$ of a $\p$-
algebra $A$, where $C^n_{\p}(A) = \Hom(\p^!(n)\ot_{S_n}A^{\ot n},A)$  (see \cite[Chapter 12]{lodayoperad} for details). We have 
that $C^*_{\mathscr Ass}(A)$ is the Hochschild cochain 
complex of $A$, $C^*_{\mathscr Com}(C)$ is the Harrison 
cochain complex of $C$, and $C^*_{\mathscr Lie}(L)$ is the Chevalley-Eilenberg cochain complex of $L$. Since we will consider  $\p$-algebras with non-trivial homological grading, the operadic cohomology will be endowed by a non-trivial homological grading, and $C^{n,p}_\p(A)$  will denote the part of  $\Hom(\p^!(n)\ot_{S_n}A^{\ot n},A)$ that is of homological degree $p\in \Z$.

The main goal of this section is to present an obstruction theory for formality in $\dga$, $\cdga$ and $\dgl$ over any field $\k$ of characteristic zero. This obstruction theory is presumably well-known for the experts, but we will recall it and formulate it in a way that is suitable for the context of this paper. In order to do that we need to recall some results by Kadeishvili (\cite{kadeishvili09}) on minimal $\ainf$algebras and the Hochschild cochain complex and minimal $\cinf$algebras and the Harrison cochain complex. The ideas of Kadeishvili apply also to minimal $\linf$algebras and the Chevalley-Eilenberg cochain complex (we leave the details to the reader).

\subsection*{Minimal $\p_\infty$-algebras}
We will now present some results by Kadeishvili in \cite{kadeishvili09}.

\begin{dfn}
 Let $\p = \mathscr Ass, \mathscr Com$, or $\mathscr Lie $. A $\p_\infty$-algebra $(H,\{b_i\})$ is called \textit{minimal} if $b_1 = 0$. 
\end{dfn}
Given a minimal $\p_\infty$-algebra $(H,0,b_2,b_3,\dots)$, we have that $\h = (H,0,b_2)$ is a $\p$-algebra, and therefore it makes sense to consider the operadic cochain complex $C^*_\p(\h)$ of $\h$.

\begin{prop}[\cite{kadeishvili09}]\label{KadResults} Suppose $\p = \mathscr Ass$, $\mathscr Com$, or $ \mathscr Lie$. Then  the following holds:
\begin{itemize}
\item[(a)] Let 
$(H,\{b_i\})$ and $(H,\{b_i'\})$ be two minimal $\p_\infty$-algebras  with $b_2=b_2'$ and let
$\phi = (\id,0,\dots,0,\phi_k,\phi_{k +1},\dots)\colon    (H,\{b_i\})\to (H,\{b_i'\})$ be a $\p_\infty$-algebra isomorphism. The formal sums 
\[\bar b = b_3+b_{4}+\dots,\qquad \bar{b}'= b'_3+b'_{4}+\dots,\qquad \bar\phi = \phi_k+\phi_{k+1}+\dots \]
in $C_\p^*(\h)$ (where $\h=(H,0,b_2)$)  satisfy the following equality
\[ \bar b-\bar{b}' = \partial_\p(\bar \phi) + (\text{elements in $C^{\geq k+2}(\h)$})\] 
($\partial_\p$ is the differential of $C^*_\p(\h)$).

\item[(b)] Let  $(H,\{b_i\})$ be a minimal $\p_\infty$-algebra, and let $\{\phi_n\in C^{n,n-2}_\p(\h)\}_{n\geq 2}$ be any collection of maps . Then there exists a minimal $\p_\infty$-algebra 
$(H,\{b_i'\})$ with $b'_2 = b_2$  such that $\phi = (\id,\phi_2,\phi_3,\dots)$ is a $\p_\infty$-algebra isomorphism $(H,\{b_i\})\to(H,\{b_i'\})$.
\end{itemize}
\end{prop}

\subsection*{Obstruction to formality} We will in the spirit of Halperin and Stasheff (\cite{halperin79}) present an obstruction theory for $\p$-algebra formality that is presumably well-known for experts. However, the author could not find in the literature an exposition that was optimized for the context of this paper. Obstructions to formality in $\cdga$ is treated in \cite{halperin79} and obstructions to formality in $\dgl$ is treated in \cite{manetti15}.
  
  We start by recalling an easy consequence of the homotopy transfer theorem for $\p_\infty$-algebras, where $\p$ is a Koszul operad.
  
  \begin{prop}\label{htt}
  Let $\p=\mathscr Ass$, $\mathscr Com$, or $\mathscr Lie$ and let $(A,\bar b_1,\bar b_2)$ be a dg $\p$-algebra. Then there exists a  $\p_\infty$-algebra structure $(H(A),0,b_2,b_3,\dots)$ on the underlying vector space of the cohomology $H(A)$ such that i) $b_2\colon    H(A)^{\ot 2}\to H(A)$ is the induced $\p$-algebra multiplication on the cohomology $H(A)$ and ii) $(A,\bar b_1,\bar b_2)$ is $\p_\infty$-quasi-isomorphic to   $(H(A),0,b_2,b_3,\dots)$
  \end{prop}
  \begin{proof}
Since $\mathscr Ass$, $\mathscr Com$, and $\mathscr Lie$ are all Koszul, the theorem follows easily from the homotopy transfer theorem for $\p_\infty$-algebras (see \cite[Section 10.3]{lodayoperad} or \cite{berglund14}).  
  \end{proof}
  \begin{rmk}
 Note that $(A,\bar b_1,\bar b_2)$ is formal if and only if there exists a $\p_\infty$-algebra quasi-isomorphism $(H(A),0,b_2,b_3,\dots)\to (H(A),0,b_2)$ (recall that quasi-isomorphisms are invertible up to homotopy in the category of $\p_\infty$-algebras). Thus, an obstruction theory for quasi-isomorphisms $(H,0,b_2,b_3,\dots)\to (H,0,b_2)$ is an obstruction theory for formality.
\end{rmk}

Now we are ready to formulate the main theorem of this section.

\begin{thm}\label{ObstrThm} Assume $\p=\mathscr Ass$, $\mathscr Com$, or $\mathscr Lie$ and that  $\mathcal H = (H,0,b_2)$  is a  dg $\p$-algebra with trivial differential. Given a $\p_\infty$-algebra of the form $(H,0,b_2,b_3,\dots)$ there is an associated sequence of cohomology classes $[b_3],[b_4'],[b_5''],\dots$, where  $[b_k^{(k-3)}]\in H_\p^{k,k-1}(\mathcal H)$. This sequence is either an infinite sequence of vanishing cohomology classes, or finite and terminating in a non-zero cohomology class $[b_k^{(k-3)}]$. There exists a $\p_\infty$-algebra quasi-isomorphism $(H,0,b_2,b_3,\dots)\to (H,0,b_2)$ if and only if $[b_3],[b_4'],[b_5''],\dots$ is an infinite sequence of vanishing cohomology classes.
\end{thm}

This theorem will follow easily from the following proposition.

\begin{prop}\label{hejhej} Assume $\p=\mathscr Ass$, $\mathscr Com$, or $\mathscr Lie$. Let $\h=(H,0,m_2)$ be a given minimal dg $\p$-algebra. 
\begin{enumerate}[label = {(\alph*)}]
\item Let $H_\alpha = (H,0,m_2,0,\dots,0,m_k,m_{k+1}\dots)$, $k\geq 3$,  be a $\p_\infty$-algebra that is quasi-isomorphic to $\h=(H,0,m_2)$. Then $m_k$ is a boundary in $C^*_\p(\h)$, i.e. $[m_k]=0$ in $H_\p^*(\h)$.\label{hejheja}
\item Given a $\p_\infty$-algebra $H_\alpha= (H,0,m_2,0,\dots,0,m_k,m_{k+1}\dots)$ where $[m_k] =0$ in $H^*_{\p}(\h)$, i.e. $m_k=\partial_\p(\phi_{k-1})$ for some $\phi_{k-1}\in C^{k-1}_\p(\h)$, then $H_\alpha$ is quasi-isomorphic to some $\p_\infty$-algebra  $H_\beta$ of the form
\[H_\beta = (H,0,m_2,0,\dots,0,m'_{k+1},m'_{k+2}\dots)\]\label{hejhejb}
             \end{enumerate}
\end{prop}

\begin{rmk}
 Note that  if all obstructions from Theorem \ref{ObstrThm} vanish, we will get a sequence of quasi-isomorphisms
 \[ (H,0,m_2,m_3,\dots)\to (H,0,m_2,0,m_4',m_5',\dots)\to (H,0,m_2,0,0,m_5'',m_6'',\dots)\to\dots\]
 One can easily see that the colimit of this diagram is $(H,0,m_2,0,\dots)$. Since quasi-isomorphisms between minimal $\p_\infty$-algebras are isomorphisms, it follows that\break $(H,0,m_2,m_3,\dots)\to (H,0,m_2,0,\dots)$ is an isomorphism (and in particular a quasi-isomorphism).
\end{rmk}

\begin{proof}
(a) By Lemma \ref{PInfMorphReduced},  there exists a morphism
\[ \phi = (\id,0,\dots,0,\phi_{k-1},\phi_{k},\dots)\colon    H_\alpha \to \h.\]

It follows from Proposition \ref{KadResults} (a) that 
\[m_k+ m_{k+1}+\dots  = (\partial_\p(\phi_{k-1})+  \partial_\p(\phi_{k})+\dots )+ (\text{elements in $C^{\geq k+1}_\p(\h)$}).\]
Collecting the elements of $C^k_\p(\h)$ from both sides of the equality gives that $m_k = \partial_\p(\phi_{k-1})$.
\\\\
(b) By Proposition \ref{KadResults} (b) there exists a $\p_\infty$-algebra $H_\beta = (H,0,m_2,m_3',m_4',\dots)$ such that
\[(\id,0,\dots,0,\phi_{k-1},0,\dots)\colon    H_\alpha \to H_\beta \]
is a $\p_\infty$-algebra isomorphism. By Proposition \ref{KadResults} (a) we have that
\[
 (m_k+m_{k+1}+\dots)- (m_3'+m_4'+\dots) = \partial_\p(\phi_{k-1})+(\text{elements in $C^{\geq k+1}_\p(\h)$})
\]
We see from the equality that $m'_3,\dots, m'_{k-1}$ vanish. We also see that $m_k-m_k' = \partial_\p(\phi_{k-1})$, giving that $m_k'=0$. This completes the proof.
\end{proof}

\section{Proof of Theorem \ref{mainthmone}}
We used the language of operadic cohomology in the obstruction theory for formality  in the previous section. We will compare different cohomology theories corresponding to different operads in order to compare the concept of formality in different categories. Recall that $H_{\mathscr Ass}^*$ and $H_{\mathscr Lie}^*$ correspond to  the Hochschild and the Chevalley-Eilenberg  cohomologies respectively. The Hochschild cochain complex of an associative algebra $A$ with coefficients in $A$ will be denoted by $C_{Hoch}^*(A)$ and its cohomology will be denoted by  $HH^*(A)$. The Chevalley-Eilenberg cochain complex of a Lie algebra $L$ with coefficients in $L$ will be denoted by $C_{CE}^*(L)$  and its cohomology will be denoted by $H_{CE}^*(L)$. We will also work with the Chevalley-Eilenberg cochain complex of a Lie algebra with coefficients in a left $L$-module $M$ different from $L$, which will be denoted by $C_{CE}^*(L,M)$ and its cohomology will be denoted by $H_{CE}^*(L,M)$.

\subsection*{Hochschild and Chevalley-Eilenberg Cohomology}
Recall that the universal enveloping algebra  $UL$ of a dg Lie algebra $L$ is explicitly given by 
\[UL = T^*_a(L)/(ab-(-1) ^{|a||b|}ba-[a,b],\quad a,b\in L).\]
A Lie algebra $L$ is of course a left module over itself by $g.h = [g,h]$.

Let $UL^{ad}$ denote the left $L$-module structure on $UL$ given by $g.m = g\ot m -(-1)^{|g||m|}m\ot g$, for $g\in L$ and $m\in UL$ (where $m$ is of some homogenous degree $|m|$). This makes the inclusion $L\hookrightarrow UL^{ad}$ a map of left $L$-modules.

\begin{lemma}[\text{\cite[Lemma 3.3.3.]{lodaycyclic}}] \label{alt}
There exists a cochain map 
\[ Alt \colon    C^*_{Hoch}(UL)\to C^*_{CE}(L,UL^{ad})\] 
from the Hochschild cochain complex of $UL$ to the Chevalley-Eilenberg cochain complex of $L$ with coefficients in $UL^{ad}$. If $f\in C^n_{Hoch}(UL) = \Hom_{\k}(UL^{\ot n},UL)$, then $Alt(f)\in C_{CE}^n(L,UL^{ad}) = \Hom_{\k}(L^{\wedge n},UL^{ad})$ is given by 
\[Alt(f)(l_1\wedge\cdots\wedge l_n) = \sum_{\sigma\in S_n} \gamma(\sigma,l_1,\dots,l_n)f(l_{\sigma^{-1}(1)}\ot\cdots\ot l_{\sigma^{-1}(n)}) \]
where $\gamma(\sigma;l_1,\dots,l_n)$ is the product of the sign of $\sigma$ and the Koszul sign obtained by applying $\sigma$ on $l_1\cdots l_j$.
\end{lemma}

By the map above we have a tool for comparison of cohomology classes in $HH^*(UL)$ and $H_{CE}^*(L,UL^{ad})$. However, the obstruction theory for formality in $\dgl$ was expressed in terms of cohomology classes in $H^*_{CE}(L)$ (i.e  $H^*_{CE}(L,L)$). In the next proposition we show that the inclusion $L\hookrightarrow UL^{ad}$ of left $L$-modules induces an injection $H^*_{CE}(L,L)\to H_{CE}^*(L,UL^{ad})$ in cohomology.

\begin{prop}\label{lie-mod-incl}
The inclusion of $L$-modules $L\hookrightarrow UL^{ad}$ induces an injection\break $H^*_{CE}(L,L)\to H_{CE}^*(L,UL^{ad})$ in cohomology.
\end{prop}

\begin{proof}
We start by recalling the Poisson algebra structure on $\Lambda_a L$ (see \cite[Section 3.3.4]{lodaycyclic}). The Poisson bracket $\{-,-\}$ on $\Lambda_a L$ is  determined by the following two properties: i) $\{g,h\} = [g,h]$ for $g,h\in L$ and ii) $\{-,-\}$ is a derivation in each variable. Now we may give $\Lambda L$ a left $L$-module structure given by $g.\alpha = \{g,\alpha\}$. With this $L$-module structure, the Poincaré-Birkhoff-Witt isomorphism $\eta\colon    \Lambda L\to UL^{ad}$ is an $L$-module morphism (\cite[Lemma 3.3.5]{lodaycyclic}). In particular, $\Lambda L$ and $UL^{ad}$ are isomorphic as $L$-modules. Since $L$ is a direct summand of the $L$-module $\Lambda L$, it follows by the $L$-module isomorphism above that $L$ is also a direct summand of $UL^{ad}$. Hence there is a projection $UL^{ad}\twoheadrightarrow L$ of $L$-modules, and therefore $\id_L$ may be decomposed as $L\hookrightarrow UL^{ad}\twoheadrightarrow L$. This in turn gives a decomposition  $\id_{H^*_{CE}(L,L)} \colon    H^*_{CE}(L,L)\to H^*_{CE}(L,UL^{ad})\to H^*_{CE}(L,L)$. Thus,  $H^*_{CE}(L,L)\to H^*_{CE}(L,UL^{ad})$ must be injective.
\end{proof}

\subsection*{The proof}
In this section it will be necessary to be able to  distinguish between a dg Lie algebra $(L,\bar l_1, \bar l_2)$ and the underlying vector space $L$. Therefore we will denote the Lie algebra structure by $\cL$ and the underlying vector space by $L$. We will denote the Lie algebra structure on the cohomology of $\cL$ by $H(\cL)$ while its underlying vector space will be denoted by $H(L)$. We make the same distinction between $U\cL$ and $UL$.

\begin{lemma}[\text{\cite[Theorem 21.7]{felixrht}}]\label{quillen} Suppose $char(\k)=0$ and $\mathcal L\in \dgl$. Then there exists a natural isomorphism $UH(\mathcal L) \cong H(U\mathcal L)$ of algebras.
\end{lemma}

It follows directly from the lemma that $U\colon   \dgl\to \dga$ preserves formality. Thus, what is left to show in order to prove Theorem \ref{mainthmone} is that whenever $U\cL$ is formal in $\dga$, then $\cL$ is formal in $\dgl$. In the language of $\ainf$ and $\linf$algebras, we need to prove the following:

\begin{thm}
Let $\k$ be a field of characteristic zero and let $\cL\in \dgl$. If there exists an $\ainf$quasi-isomorphism $U\cL\to H(U\cL)$ then there exists an $\linf$quasi-isomorphism $\cL\to H(\cL)$. 
\end{thm}

\begin{proof}Let $\mathcal L$ be on the form $\cL = (L,\bar l_1,\bar l_2)$ and let 
$U\mathcal L = (UL,\bar m_1,\bar m_2)$ be its  universal enveloping algebra. 

By the homotopy transfer theorem for $\linf$algebras (Proposition \ref{htt}) there exists an $\linf$algebra
$(H^*(L),0,l_2,l_3,\dots)$ and an $\linf$quasi-isomorphism 
\begin{equation*}
\phi\colon   \mathcal L\to(H(L),0,l_2,l_3,\dots).\end{equation*}
Applying $\ubar$ to $\phi$ gives an $\ainf$quasi-isomorphism
\begin{equation}\label{ubar-phi}\ubar(\phi) \colon    \ubar(\mathcal L)\to \ubar(H(L),0,l_2,l_3,\dots)\end{equation}
(recall from \hyperref[ubar]{Theorem \ref*{ubar} (b)} that $\ubar$ preserves quasi-isomorphisms). By \hyperref[ubar]{Theorem \ref*{ubar} (e)}, $\ubar(\mathcal L)$ is the ordinary universal enveloping algebra $U\mathcal L$ of $\mathcal L$. Let us analyse the $\ainf$structure of $\ubar(H^*(L) ,0,l_2,l_3,\dots)$.
\\\\
 \textit{\textbf{Claim:} $\ubar(H(L),0,l_2,l_3,\dots)$ is an $\ainf$algebra  $(H(UL),0,m_2,m_3,\dots)$, whose 2-truncation, $(H(UL),0,m_2)$, is  isomorphic to the cohomology algebra $H(U\mathcal L)$ of the universal enveloping algebra $U\mathcal L$.}
\begin{proof} Since $\ubar|_{\dgl} = U$, the following holds:  \begin{equation*}\label{UswitchH}\ubar(H(L),0,l_2) = UH(\mathcal L)
\overset{\text{Lemma \ref{quillen}}} = H(U\mathcal L)=H(UL,\bar m_1,\bar m_2)=(H(UL),0,m_2)\end{equation*}
Now it follows by \hyperref[ubar]{Theorem \ref*{ubar} (c)}  that
\[\ubar(H(L),0,l_2,l_3,\dots)= (H(UL),0,m_2,m_3,\dots).\]
\end{proof}

Thus the quasi-isomorphism in \eqref{ubar-phi} is a map of the form
$$\ubar(\phi)\colon    (UL,\bar m_1,\bar m_2)\to (H^*(UL),0,m_2,m_3,\dots)
$$
Since $U\cL$ is formal, it follows  that $(H(UL),0,m_2,m_3,\dots)$ is $\ainf$quasi-isomorphic to $H(U\cL )=(H(UL),0,m_2)$. It follows by \hyperref[hejhej]{Proposition \ref*{hejhej} (a)} that  $[m_3] =0$ in $HH^*(UH(\mathcal L))$. 

Let $Alt^*\colon    HH^*(UH(\mathcal L))\to H^*_{CE}(H(\mathcal L),(UH(\mathcal L))^{ad})$ be the cohomology-induced map of the cochain map $Alt$ introduced in \hyperref[alt]{Lemma \ref*{alt}} and let $j^*\colon    H^*_{CE}(H(\mathcal L),H(\mathcal L))\to H^*_{CE}(H(\mathcal L),(UH(\mathcal L))^{ad})$ be the cochain map induced by the inclusion $H(\mathcal L)\hookrightarrow (UH(\mathcal L))^{ad}$. We have by Theorem \ref{ubar} (d) that $Alt^*[m_3] = j^*[l_3]$. Since $[m_3] = 0$, it follows that $j^*[l_3]= 0$. By Proposition \ref{lie-mod-incl}, $j^*$ is injective, and hence $[l_3]=0$ in $H^*_{CE}(H(\cL),H(\cL))$. 

Since $[l_3]=0$, it follows by Proposition \ref{hejhej} (b), that there exists a quasi-isomorphism 
$$\alpha\colon    (H(L),0,l_2,l_3,\dots)\to (H(L),0,l_2,0,l_4',l_5',\dots)$$

Applying Theorem \ref{ubar} (c) on $\ubar(H(L),0,l_2,0,l_4',l_5',\dots)$ and $\ubar(H(L),0,l_2,0,\dots)$ we get that  
\[\ubar(H(L),0,l_2,0,l_4',l_5',\dots)=(H^*(UL),0,m_2,0,m_4',m_5',\dots).\]

Note that $(H(UL),0,m_2,0,m_4',m_5',\dots)$ is $\ainf$quasi-isomorphic to $(H(UL),0,m_2)$ (since $\ubar(\alpha)$ is a quasi-isomorphism (Theorem \ref{ubar} (b)) connecting $(H(UL),0,m_2,m_3,\dots)$ and $(H(UL),0,m_2,0,m_4',m_5',\dots)$). Again, by Proposition \ref{hejhej} (a) it follows that $[m_4']=0$ in $HH^*(UH(\cL))$. The same reasoning as before will give us that $[l_4']=0$ in $C^*_{CE}(H(\cL))$. 

Continuing this process will yield a sequence
\[ [l_3], [l_4'],\dots, [l_n^{(n-3)}],\dots\]
of vanishing Chevalley-Eilenberg cohomology classes. By Theorem \ref{ObstrThm}, it follows that $(H(L),0,l_2,l_3,\dots)$ is $\linf$quasi-isomorphic to $(H(L),0,l_2)$, which is equivalent to the $\dgl$-formality of $\cL = (L,l_1,l_2)$.
\end{proof}

\section{Proof of Theorem \ref{mainthmtwo}}
We will compare the cohomology theories $H_{\mathscr Ass}^*$ and $H_{\mathscr Com}^*$, which correspond to the Hochschild and the Harrison cohomologies respectively, in order to compare the concept of formality in $\dga$ and $\cdga$. We will denote the Harrison cochain complex and the Harrison cohomology of a commutative dg algebra $A$ with coefficients in $A$ by $C_{Harr}^*(A)$ and $Harr^*(A)$ respectively.

\subsection*{Hochschild and Harrison Cohomology}
We will start by recalling the notion of shuffle products. A permutation $\sigma\in S_{p+q}$ is called a $(p,q)$-shuffle if $\sigma(1)<\cdots <\sigma(p)$ and $\sigma(p+1)<\cdots<\sigma(p+q)$. Let $\mu_{p,q}\in\k[S_{p+q}]$ be given by 
\[\mu_{p,q} = \sum_{(p,q)\text{-shuffles}}sgn(\sigma)\sigma\]

There is an action of $\k[S_n]$ on $A^{\ot n}$ given by \[\sigma.(a_1\cdots a_n) = \epsilon(\sigma;a_1,\dots,a_n)a_{\sigma^{-1}(1)}\cdots a_{\sigma^{-1}(n)}\] for $\sigma\in S_n$ and where $\epsilon(\sigma;a_1,\dots,a_n)$ is 
the Koszul sign obtained when applying $\sigma$ on $a_1\cdots a_n$. The shuffle product $\bar \mu_{p,q}\colon    A^{\ot p}\ot A^{\ot q}\to A^{\ot(p+q)}$ is given by letting $\mu_{p,q}$ act on $A^{\ot p}\ot A^{\ot q}\cong A^{\ot(p+q)}$.

We will now see how this is related to Harrison cohomology. We have that 
\[C_{Harr}^n(A) \cong C_{\mathscr Com}^n(A) \cong \Hom_{\k}(\mathscr Com^!(n)\ot_{S_n} A^{\ot n},A)\cong \Hom_{\k}(\mathscr Lie(n) \ot_{S_n} A^{\ot n},A). \] Over a field of characteristic zero one can show that  
$\Hom_{\k}(\mathscr Lie(n) \ot_{S_n} A^{\ot n},A)$  is isomorphic to the space of $\k$-morphisms $A^{\ot n}\to A$ that vanish on all shuffle products $\bar\mu_{k,n-k}\colon    A^{\ot k}\ot A^{\ot (n-k)}\to A^{\ot n}$ (see \cite[Sections 1.3.3, 13.1.7.]{lodayoperad}) In particular that means that there exists an inclusion 
$$\iota\colon    C^*_{Harr}(A)\hookrightarrow C^*_{Hoch}(A)\cong\Hom_{\k}(A^{\ot n},A).$$

This inclusion induces a map $\iota^*\colon    Harr^*(A)\to HH^*(A)$ in the cohomology. Over a field $\k$ of characteristic zero Barr showed in \cite{barr68} that $\iota^*$ is injective.
\begin{prop}[\cite{barr68}]\label{barr-thm}
 Let $\k$ be a field of characteristic zero, and let $A$ be a commutative dg algebra over $\k$. The map
$ \iota^*\colon    Harr^*(A)\to HH^*(A)$ induced by the inclusion $\iota\colon    C^*_{Harr}(A)\to C^*_{Hoch}(A)$ is injective.
\end{prop}

We will briefly explain the techniques used in the proof of the proposition above.
First, set $\mu_n = \sum_{i=1}^{n-1}\mu_{i,n-i}$. Next, Barr constructed a family of idempotents  $\{e_i\}_{i\geq 2}$, $e_n\in \k[S_n]$, that satisfies the following conditions
\begin{itemize}
\item[(i)] $e_n^2 = e_n$ (idempotent)
\item[(ii)] $e_n$ is a polynomial in $\mu_n$ (without any constant term)
\item[(iii)] $e_n\mu_{i,n-i} = \mu_{i,n-i}$,\quad $1\leq i \leq n-1$
\end{itemize}
Since $e_n \in \k[S_n]$, it defines an action on $C^n_{Hoch}(A)=\Hom(A^{\ot n},A)$ (by permuting the inputs). This allows us to formulate a fourth condition that $\{e_i\}$ satisfies:
\begin{itemize}
\item[(iv)] $\partial_{Hoch}e_n = e_{n+1}\partial_{Hoch}$\qquad ($\partial_{Hoch}$ is the Hochschild coboundary).
\end{itemize}
By conditions (ii) and (iii) there is an equality of ideals  $(e_n) = (\mu_{1,n-1},\mu_{2,n-2},\dots,\mu_{n-1,1})$. In particular, a map $\phi\in C^n_{Hoch}(A)=\Hom(A^{\ot n},A)$ vanishes by the action of $e_n$ if and only if it vanishes on all $\mu_{i,n-i}$ (which is equivalent to that $\phi\in C^*_{Harr}(A)$). 

Recall that  an endomorphism $\rho\colon    V\to V$ gives a decomposition $V = \ker(\rho)\oplus \im(\rho)$. If $\rho$ is an idempotent we have that $\rho(a,b)= (0,b)$. Applying this to $e_n$ (which defines an endomorphism on $C^n_{Hoch}(A)$), we get that 
$C^n_{Hoch}(A) = \ker(e_n)\oplus\im(e_n)$. Since $(e_n)  = (\mu_{1,n-1},\mu_{2,n-2},\dots,\mu_{n-1,1})$, it follows that $\ker(e_n) = C^n_{Harr}(A)$. Set $W^n(A)\colon   =\im(e_n)$ and $C^n_{Hoch}(A)$ is then decomposed as 
\begin{equation}\label{hhdecomp}C^n_{Hoch}(A) \cong C^n_{Harr}(A)\op W^n(A).\end{equation}

In order to show that $\iota^* \colon    Harr^*(A)\to HH^*(A)$ is injective, we have to show that if $x\in C^n_{Harr}(A)\subset C^n_{Hoch}(A)$ is a coboundary in $C^*_{Hoch}(A)$ then it is also a coboundary in $C^*_{Harr}(A)$. 
By \eqref{hhdecomp} we have that an element of the Harrison subcomplex may be represented by an element of the form $(x,0)\in C^n_{Harr}(A)\op W^n(A) \cong C^n_{Hoch}(A)$. Assume $(x,0)$  is a coboundary in $C^n_{Hoch}(A)$, meaning that there is some element $(y_1,y_2)\in  C^{n-1}_{Harr}(A)\op W(A) \cong  C^{n-1}_{Hoch}(A)$ such that 
$\partial_{Hoch}(y_1,y_2) = (x,0)$.  From property (iv) we get the following commutative diagram
$$
\xymatrix{(y_1,y_2)\ar[r]^{e_{n-1}}\ar[d]_{\partial_{Hoch}} &(0,y_2)\ar[d]^{\partial_{Hoch}}\\
(x,0)\ar[r]_{e_n}& (0,0).}$$
Now we see that $\partial_{Hoch}(y_1,0) = \partial_{Hoch}((y_1,y_2)-(0,y_2)) = (x,0)$, which proves that $(x,0)$ is also a boundary in $C^*_{Harr}(A)$. This proves that $\iota^*$ is injective. 

The idempotents $\{e_n\}$ can not be constructed over a field of characteristic $p>0$ and over such a field $\iota^*$ is not injective in general (see the example in Section 4 in \cite{barr68}).

We would like to remark that the overview above is related to the subject of the $\lambda$-decomposition of Hochschild homology (see \cite[Section 4.5]{lodaycyclic}).

\subsection*{The proof}
As mentioned in the introduction, it is obvious that $\cdga$-formality implies $\dga$-formality. Hence what is left to show in order to prove Theorem \ref{mainthmtwo} is that if a cdga is formal as a dga, then it is also formal as a cdga.

\begin{proof}[Proof of Theorem \ref{mainthmtwo}]
Let $(C,\bar m_1,\bar m_2)$ be a cdga that is formal in $\dga$. Let $\mathcal H = (H(C),0,m_2)$ be the induced commutative graded algebra structure on the cohomology of $(C,\bar m_1,\bar m_2)$. The homotopy transfer theorem for $\cinf$algebras (see Proposition \ref{htt}) gives that there exists a $\cinf$algebra $(H(C),0,m_2,m_3,\dots)$ equipped with a $C_\infty$-quasi-isomorphism 
\[(C,\bar m_1,\bar m_2)\to (H(C),0,m_2,m_3,\dots).\]
Since $C$ is formal in $\dga$ there exists an $\ainf$quasi-isomorphism 
\[(H(C),0,m_2,m_3,\dots)\to (H(C),0,m_2).\]
It follows by \hyperref[hejhej]{Proposition \ref*{hejhej} (a)},  that $[m_3]_{Hoch}= 0$ in $HH^*(\h)$. Since the cohomology map $\iota^*\colon    Harr^*(A)\to HH^*(A)$ induced by the inclusion $\iota\colon    C_{Harr}^*(H(C))\hookrightarrow C_{Hoch}^*(H(C))$ is injective (Propostion \ref{barr-thm}), it follows that $[m_3]_{Harr}=0$ in $Harr^*(\h)$. Now by \hyperref[hejhej]{Proposition \ref*{hejhej} (b)} it follows that there exists a $\cinf$quasi-isomorphism
\[(H(C),0,m_2,m_3,\dots)\to(H(C),0,m_2,0,m_4',m_5',\dots).\]
Applying \hyperref[hejheja]{Proposition \ref*{hejhej} \ref*{hejheja}} again, gives that $[m_4']_{Hoch}=0$ in $HH^*(\h)$, wich in turn gives together with the injectivity of $\iota^*$ that $[m_4']_{Harr}=0$ in $Harr^*(\h)$. Continuing this process will yield a sequence \[[m_3]_{Harr}, [m_4']_{Harr},\dots,[m^{(n-3)}_n]_{Harr},\dots\] of vanishing Harrison cohomology classes in $Harr^*(\h)$. By \hyperref[ObstrThm]{Theorem \ref*{ObstrThm}}, it follows that $(H(C),0,m_2, m_3,\dots)$ is $\cinf$quasi-isomorphic to $(H(C),0,m_2)$, which is equivalent to the $\cdga$-formality of $(C,\bar m_1,\bar m_2)$. 
\end{proof}

\begin{appendices}

\section{Some technicalities concerning $\ainf$, $\cinf$ and $\linf$algebras} 
Given  a Koszul operad $\p$ there are many equivalent ways of viewing a  $\p_\infty$-algebra structure 
 on a vector space $A$.

 \begin{thm}[{\cite[Theorem 10.1.13]{lodayoperad}}] 
 Let $\p$ be a Koszul operad. Then a $\p_\infty$-algebra structure on a vector space $A$ is the same thing 
as coderivation on the cofree $\p^!$-coalgebra on $sA$, denoted by $\Vpd^*(sA)$ and a morphism of $\p_\infty$-algebras is the same thing as a morphism of cofree $\p^!$-coalgebras \end{thm}
 
We have that $\Vpd^*(sA) = \bigoplus_{n\geq 0} \Vpd^n(sA)$ where $\Vpd^n(sA) = {\p^!}^\vee(n)\otimes_{S_n} (sA)^{\ot n}$ (and where ${\p^!}^\vee$ denotes the cooperad obtained by dualising  $\p^!$). We say that an element of $\Vpd^n(sA)$ is of word-length $n$. 

We will briefly recall the correspondence between $\p_\infty$-algebras and quasi-free dg $\p^!$-coalgebras. A $\p^!$-coalgebra differential $d$ on $\Vpd^*(sA)$ may be decomposed as
\[ 
d = d_0+d_1+\dots,
 \]
where $d_i$ is the part of $d$ that decreases the word-length  by $i$. The dg coalgebra\break $(\Vpd^*(sA),d = d_0+d_1+\dots)$ corresponds to a $\p_\infty$-algebra $(A,b_1,b_2,\dots)$ where $d_i$ and $b_{i+1}$ encodes each other (i.e. $d_i$  may be constructed from $b_{i+1}$ and vice versa).

Analogously, a  morphism of  dg $\p^!$-coalgebras $\Psi\colon   (\Vpd^*(sA),d)\to (\Vpd^*(sA'),d')$ may be decomposed as $\Psi = \Psi_0+\Psi_1+\dots$, where $\Psi_i$ is the part of $\Psi$ that decreases the word-length by $i$. We have that $\Psi$ corresponds to a $\p_\infty$-quasi-isomorphism $\phi = (\phi_1,\phi_2,\dots)\colon    A\to A'$ where $\Psi_i$ and $\phi_{i+1}$ encodes each other. With this correspondence we have tools to prove some technical results that we  need in this paper. The author was inspired by the techniques used in \cite[Section 2.72]{felixamg}.

\begin{lemma}\label{appendixforsta}
Assume $(A,0,b_2,b_3,\dots)$ and $(A,0,b_2)$ are quasi-isomorphic as $\p_\infty$-algebras. Then there exists a $\p_\infty$-algebra quasi-isomorphism $\phi'\colon   (A,0,b_2,b_3,\dots)\to (A,0,b_2)$ where $\phi_1'=\id_A$
\end{lemma}

\begin{proof} Let $\phi$ be a quasi-isomorphism $\phi= (\phi_1,\phi_2,\dots)\colon    (A,0,b_2,b_3,\dots)\to (A,0,b_2)$. Since $\phi$ is a quasi-isomorphism of  minimal $\p_\infty$-algebras, it follows that $\phi$ is an isomorphism.

We have that $\phi$ corresponds to a map 
\[\Psi = \Psi_0+\Psi_1+\dots\colon    (\Vpd^*(sA),d_1+d_2+\dots)\to (\Vpd^*(sA),d_1),\] where $\Psi_i$ increases the word-length by $i$ and corresponds to $\phi_{i+1}$. Moreover, $\Psi_0$ is a vector space isomorphism (since $\Psi$ is a dg $\p^!$-coalgebra isomorphism).

We show that $\Psi_0$ commutes with the differential $d_1$. Since $\Psi$ is a chain map, we have that
\[(\Psi_0+\Psi_1+\dots)\circ (d_1+d_2+\dots) = d_1\circ (\Psi_0+\Psi_1+\dots).\]
Collecting the terms that decreases the word-length by 1 from both sides of the equality gives that $\Psi_0d_1= d_1\Psi_0$.

Similar techniques gives also that $\Psi_0$ commutes with the comultiplication $\Delta$ on $\Vpd^*(sA)$. Hence $\Psi_0\colon    (\Vpd^*(sA),d_1)\to(\Vpd^*(sA),d_1)$ is a dg $\p^!$-coalgebra automorphism which has an inverse $\Psi_0^{-1}$. Now the composition $\Psi' = (\Psi^{-1}_0)\circ (\Psi_0+\Psi_1+\dots)$ will give the desired result.
\end{proof}

\begin{lemma}\label{expofcoder}
Assume that   $\theta$ is a $\p^!$-coalgebra coderivation on $\Vpd^*(V)$ of cohomological degree 0 that decreases the word-length by some number $i\geq 1$. Then the map
$$e^\theta = \id + \theta + \frac{\theta^2}{2!}+ \frac{\theta^3}{3!}+\dots$$
is a well defined map of $\p^!$-coalgebras.
\end{lemma}

\begin{proof}
 The map is well defined since for any element $x\in \Vpd^*(V)$ of word-length $k$ we have that $\theta^m(x)=0$ for all $m\geq\left\lceil\frac{k}{i} \right\rceil$, so $e^\theta(x)$ will be a finite sum
 $$e^\theta(x) = x+\theta(x)+\dots+\frac{\theta^{m-1}(x)}{(m-1)!}.$$
 Now we prove that $e^\theta$ is a map of $\p^!$-coalgebras. One can easily prove by induction that 
 $$\Delta \theta^n = \left(\sum_{p=0}^n\binom{n}{p}\theta^{n-p}\ot\theta^p \right)\circ \Delta$$
 Thus 
 \begin{align*}
  \Delta\circ e^\theta & = \left(\sum_{n=0}^\infty\frac{1}{n!}\sum_{p=0}^n\binom{n}{p}\theta^{n-p}\ot\theta^p\right)\circ \Delta
  \\ & = \left(\sum_{n=0}^\infty\sum_{p=0}^n\frac{\theta^{n-p}}{(n-p)!}\ot\frac{\theta^p}{p!}\right)\circ \Delta
  \\ & = (e^\theta\ot e^\theta)\circ \Delta
 \end{align*}\end{proof}

\begin{rmk}
Note that $e^\theta$ is an automorphism with inverse $e^{-\theta}$.
\end{rmk}

\begin{lemma}\label{PInfMorphReduced}
Assume $ (A,0,m_2,0,\dots,0,m_n,m_{n+1},\dots)$ and $(A,0,m_2)$ are quasi-\break isomorphic as $\p_\infty$-algebras.
Then there exists a map 
\[\phi'\colon   (A,0,m_2,0,\dots,0,m_n,m_{n+1},\dots)\to(A,0,m_2)\]
such that  $\phi_1'=\id_A$ and  $\phi_i' = 0$ for $2\leq i\leq n-2$.
\end{lemma}

\begin{proof}
We prove the lemma by induction on $n$. For $n=3$ the assertion is true by Lemma \ref{appendixforsta}.
Assume the assertion is true for $n-1$, $n\geq 4$. Then we have that there exists a quasi-isomorphism 
\[\phi = (\id,0,\dots,0,\phi_{n-2},\phi_{n-1},\dots)\colon    (A,0,m_2,0,\dots,0,m_n,m_{n+1},\dots)\to(A,0,m_2).\]
which corresponds to a $\p^!$-coalgebra map
\[\Psi = \id + \Psi_{n-3}+\Psi_{n-2}+\dots\colon    (\Vpd^*(sA),d_1+d_{n-1}+d_n+\dots)\to (\Vpd^*(sA),d_1)\]
Considering the equality $(\Psi\ot\Psi)\circ \Delta = \Delta \circ \Psi$ and collecting the terms that decreases the word-length by $n-3$ gives that $(\id\ot\Psi_{n-3}+\Psi_{n-3}\ot\id)\circ \Delta= \Delta\circ\Psi_{n-3}$. That means that $\pm\Psi_{n-3}$ is a  coderivation of $\Vpd^*(sA)$ and therefore $e^{\pm \Psi_{n-3}}\colon    \Vpd^*(sA)\to \Vpd^*(sA)$ is a $\p^!$-coalgebra automorphism. 

Considering the equality $\Psi\circ (d_1+d_{n-1}+d_n+\dots) = d_1\circ \Psi$ and collecting the terms that decreases the word-length by  $n-2$ gives that $\Psi_{n-3}\circ d_1 = d_1\circ \Psi_{n-3}$, i.e. that $\pm\Psi_{n-3}$ commutes with the differential $d_1$. Hence $e^{\pm \Psi_{n-3}}$ commutes with $d_1$ and therefore $e^{\pm \Psi_{n-3}}\colon    (\Vpd^*(sA),d_1) \to (\Vpd^*(sA),d_1)$ is  a dg $\p^!$-coalgebra automorphism. 

We consider the composition $\Psi' = e^{-\Psi_{n-3}}\circ \Psi\colon    (\Vpd^*(sA),d_1+d_{n-1}+d_n+\dots)\to (\Vpd^*(sA),d_1)$. We have that
\begin{align*}
\Psi' &= e^{-\Psi_{n-3}}\circ (\id +\Psi_{n-3} + \Psi_{n-2}+\dots)\\
& = (\id-\Psi_{n-3} + \frac{\Psi_{n-3}^2}{2!}-\cdots)\circ (\id +\Psi_{n-3} + \Psi_{n-2}+\dots)\\
& = \id+ \text{(terms that increases the word-length by $\geq n-2$)}.
\end{align*} 
Hence $\Psi'$ is of the form $\Psi' = \id + \Psi'_{n-2}+\Psi_{n-1}'+\dots$ where $\Psi'_i$ decreases the word-length by $i$ and will therefore correspond to a $\p_\infty$-algebra quasi-isomorphism $\phi' \colon\break (A,0,m_2,0,\dots,0,m_n,m_{n+1},\dots)\to(A,0,m_2)$ that satisfies the property given in the lemma.\end{proof}

\end{appendices}

\bibliographystyle{amsalpha}
\bibliography{references}
\noindent
\Addresses

\end{document}